\documentclass[preprint]{elsarticle}
\usepackage{amssymb}
\usepackage{amsmath}
\usepackage{amsthm}
\usepackage{amsfonts}
\usepackage{mathrsfs}
\usepackage{hyperref}

\usepackage{latexsym}
\usepackage{tikz}

\newtheorem{theorem}{Theorem}[section]
\newtheorem{lemma}[theorem]{Lemma}
\newtheorem{definition}[theorem]{Definition}

\newtheorem{corollary}[theorem]{Corollary}
\theoremstyle{remark}
\newtheorem{example}[theorem]{Example}
\newtheorem{remark}[theorem]{Remark}

\newcommand{\R}{\mathbb R}

\begin{document}

\title{The Geometric Structure of Max-Plus Hemispaces}

\author[de]{Daniel Ehrmann}
\ead{daniel.ehrmann@spartans.ut.edu}

\author[zh]{Zach Higgins}
\ead{zhiggins11@yahoo.com}

\author[rvt1]{Viorel Nitica}
\ead{vnitica@wcupa.edu}


\address[de]{Department of Mathematics, University of Tampa, 401 West Kennedy Blvd., Campus Box 2928, Tampa, FL 33606}

\address[zh]{Department of Mathematics, University of Florida, 358 Little Hall, PO Box 118105, Gainesville, FL 32611}

\address[rvt1]{Department of Mathematics, West Chester University,
PA 19383, USA, and Institute of Mathematics, P.O. Box 1-764,
Bucharest, Romania}

\begin{abstract} Given a set $S$ endowed with a convexity structure, a hemispace is a convex subset of $S$ which has convex complement. We recall that $\R^n_{\max}$ is a semimodule over the max-plus semifield $(\R_{\max}:=\R\cup\{-\infty\},\max,+).$ A convexity structure of current interest is provided by $\R^n_{\max}$  naturally endowed with the max-plus (or tropical) convexity. In this paper we provide a geometric description of a max-plus hemispace in $\R^n_{\max}$. We show that a max-plus hemispace has a conical decomposition as a finite union of disjoint max-plus cones. These cones can be interpreted as faces of several max-plus hyperplanes. Briec-Horvath proved that the closure of a max-plus hemispace is bounded by a max-plus hyperplane. Given a hyperplane, we give a simple condition for the assignment of the faces between a pair of complementary max-plus hemispaces. Our result allows for counting and enumeration of the associated max-plus hemispaces. We recall that an $n$-dimensional max-plus hyperplane is called strictly affine and nondegenerate if it has a linear equation that contains all variables $x_1,x_2,\dots,x_n$ and a free term. We prove that the number of max-plus hemispaces in $\R^n_{\max}$, supported by strictly affine nondegenerate hyperplanes centered in the origin, is twice the $n$-th ordered Bell number. Our work can be viewed as a complement to the recent results of Katz-Nitica-Sergeev, who described generating sets for max-plus hemispaces, and the results of Briec-Horvath, who proved that closed/open max-plus hemispaces are max-plus closed/open halfspaces.
\end{abstract}

\begin{keyword}
convex structure, tropical convexity, max-plus convexity, abstract convexity,
max-plus algebra, max-plus hemispace, max-plus semispace, max-plus halfspace, max-plus hyperplane, ordered Bell numbers, max-plus cone, conical decomposition \vskip0.1cm {\it{AMS Classification:}} 15A80, 52A01, 16Y60.
\end{keyword}

\date{}
\maketitle

\section{Introduction}

The \emph{max-plus semifield} is the set $\R_{\max} = \R \cup \{-\infty\}$ endowed with the operations $\oplus =$ max, $\otimes = +$. The zero element for $\oplus$ is $-\infty$ and the identity for $\otimes$ is $0$.

The \emph{max-plus semimodule} is the set $\R^n_{\max}$ endowed with the operations of addition and scalar multiplication given by
\begin{equation*}
\begin{gathered}
x \oplus y = (x_1 \oplus y_1,...,x_n \oplus y_n)\\
\alpha \otimes x = \alpha x=(\alpha x_1,...,\alpha x_n),
\end{gathered}
\end{equation*}
where $x=(x_1,\dots,x_n),y=(y_1,\dots, y_n)\in \R^n_{\max}$ and $\alpha\in \R_{\max}$. In order to simplify the notation, we denote the element $(-\infty,-\infty,\dots,-\infty)\in \R^n_{\max}$ by $-\infty$ as well.

For a positive integer $n$, we introduce the notation $[n]=\{1,2,\dots,n\}$.

One may introduce on $\R_{\max}^n$ the topology induced by the metric $$d_{\infty}(x,y)=\max_{i\in [n]} |e^{x_i}-e^{y_i}|, x,y\in \R_{\max}^n.$$

The max-plus semimodule has a natural convex structure, called \emph{max-plus convexity} or \emph{tropical convexity}, which is of strong current interest as a mathematical object in itself, but also due to a wide range of applications in algebraic geometry, optimization, control theory, economics, computer sciences, biology, finance and many other fields of contemporary scientific research. We refer to~\cite{Zim-77, DS, BCOQ, GK-06} and to the references mentioned there for the details about this topic that are beyond the scope of the paper.

A natural and convenient way to introduce max-plus convexity is by the aim of max-plus segments~\cite{Zim-77, NS1}.

\begin{definition}
The \emph{(max-plus) segment} joining the points $x,y \in \R^n_{\max}$ is the set:
\begin{equation}
\begin{aligned}
\ [x,y]&=\{\alpha x \oplus \beta x|\alpha, \beta\in \R_{\max}, \alpha \oplus \beta =0\} \\
&= \{\max(\alpha +x, \beta +y)|\alpha, \beta\in \R_{\max}, \max(\alpha, \beta)=0\}.
\end{aligned}
\end{equation}
\end{definition}

\begin{definition}
A subset $S \subseteq \R^n_{\max}$ is said to be \emph{(max-plus) convex} if $[x,y] \subseteq S$ for all $x,y \in S$.
\end{definition}

We recall several classes of convex sets that will be used in the sequel.

\begin{definition} A subset $S \subseteq \R^n_{\max}$ is a \emph{(max-plus) semispace} at $z \in \R^n_{\max}$ if $S$ is a maximal convex subset of $\R^n_{\max}$ avoiding $z$.
\end{definition}

\begin{remark} Semispaces are introduced in~\cite{NS1, NS2}. It is shown in~\cite{NS1} that semispaces form an intersectional basis for the collection of convex sets.
\end{remark}

\begin{definition} A set $C \subseteq \R_{max}^n$ is called a \emph{(max-plus) cone} if it is closed under addition and multiplication by scalars different from $-\infty$.
\end{definition}

\begin{remark} Our notion of cone is a bit different then the usual one, as we do not allow multiplication by the scalar $-\infty$. We refer to Butkovic-Schneider-Sergeev~\cite{BSS} for an introduction to the usual (max-plus) cones.
\end{remark}

\begin{definition}
A \emph{(max-plus) closed halfspace} is the set of solutions of a max-plus linear inequality:
\begin{equation}\label{e:halfs1}
\left\{ x\in\R_{\max}^n \mid  \bigoplus_{i\in I} \beta_ix_i\oplus\delta \leq\bigoplus_{j\in J} \gamma_jx_j\oplus \alpha \makebox{ and }
x_\ell=-\infty \makebox{ for } \ell\in L \right\},
\end{equation}
where $I$, $J$ and $L$ are pairwise disjoint subsets of $[n]$, $\alpha, \delta\in \R_{\max},\beta_i, \gamma_i\in \R$, and at most one of $\alpha,\delta$ is different from $-\infty$. We call the \emph{boundary of the closed halfspace} the subset of the halfspace for which one has equality in~\eqref{e:halfs1}.
\end{definition}

\begin{remark} Note that the union $I\cup J\cup L$ can be a proper subset of $[n]$.
\end{remark}

\begin{remark} It is well known, see e.g.~\cite{BH, NS3}, that up to a finite translation and a permutation of the variables $x_i$, the boundary of a closed halfspace that contains $-\infty$ has the equation:
\begin{equation}\label{e:-boundaryhalfs1}
\left\{ x\in\R_{\max}^n \mid \bigoplus_{i\in I}x_i=\bigoplus_{j\in J}x_j\oplus \alpha \makebox{ and }
x_\ell=-\infty \makebox{ for } \ell\in L \right\},
\end{equation}
where $I, J, L$ are disjoint subsets of $[n], I\not =\emptyset, J,L$ can be empty, and $\alpha=0$ or missing, in which case $J\not = \emptyset$. Thus the boundary of a halfspace is described by a max-plus linear equation, and hence it is a \emph{(max-plus) hyperplane}. As any hyperplane is a convex set, the boundary is a convex set.
\end{remark}

We recall some terminology introduced in~\cite{NS3}.

\begin{definition} A hyperplane $\mathcal{H}$ given by~\eqref{e:-boundaryhalfs1} is called \emph{strictly affine} if $\alpha=0$ and it is called \emph{nondegenerate} if $L=\emptyset$. If $\mathcal{H}$ is strictly affine and nondegenerate, we call the origin in $\R^n_{\max}$ the \emph{center} of $\mathcal{H}$. Equivalently, $\mathcal{H}$ is said to be \emph{centered in the origin}.
\end{definition}

\begin{definition} A set $H \subseteq \R^n_{\max}$ is a \emph{(max-plus) hemispace} in $\R^n_{\max}$ if both $H$ and its complement $\complement H$ are convex. If $H$ is a hemispace, we call $(H, \complement H)$ a \emph{pair of complementary hemispaces}.
\end{definition}

Hemispaces also appear in the literature under the name of halfspaces, convex halfspaces,
and generalized halfspaces. Usual hemispaces in the linear space $\R^n$ are described by Lassak in~\cite{Lassak-84}.
Mart\'{\i}nez-Legaz and Singer~\cite{MLegSin-84} give several geometric characterizations of usual hemispaces in $\R^n$
with the aid of linear
operators and lexicographic order in $\R^n$. Hemispaces play an important role in abstract convexity (see Singer~\cite{Sin:97}, Van de Vel~\cite{VdV}),
where they are used in the Kakutani Theorem to separate
two convex sets from each other. The proof of the Kakutani Theorem makes use of
Zorn's Lemma (relying on the Pasch axiom, which holds both in tropical~\cite{Zim-77} and usual convexity). As general convex sets in a convexity structure can be quite complicated, a clear description of the hemispaces is highly desirable. Hemispaces can also be used in the investigation of more complex convex sets, such as (max-plus) polyhedra, and convex or polyhedral decompositions in multiple pieces.

In this paper we determine the geometric structure of hemispaces in $\R^n_{\max}$. Our work can be viewed as a complement to the results of Briec-Horvath~\cite{BH}, who proved that closed/open hemispaces are closed/open halfspaces, and to those of Katz-Nitica-Sergeev~\cite{KNS}, who described generating sets for hemispaces. The approach here is more elementary, with combinatorial and geometric flavor. In particular, we obtain a conical decomposition of a hemispace, see Theorem~\ref{t:main}, as a finite union of disjoint cones. Our proofs are completely independent of~\cite{KNS}, but assume~\cite{BH} as the starting point of the investigation.

Briec-Horvath~\cite{BH} show that the closure of a hemispace $H$ that contains $-\infty$ is a closed halfspace given by ~\eqref{e:halfs1}. We will refer to $\mathcal{H}$, the boundary of the halfspace, as the \emph{bounding hyperplane}, or the \emph{boundary} of the halfspace, and we will refer to  $(H,\complement H)$ as a \emph{complementary pair of hemispaces related to $\mathcal{H}$}. To simplify the notation, we assume in the future that the hyperplane is described by~\eqref{e:-boundaryhalfs1}.

The hemispaces in a complementary pair have both nonempty interior if and only if $L=\emptyset$. If so, then the interiors of a pair $(H_1,H_2), -\infty\in H_1,$ of complementary hemispaces bounded by $\mathcal{H}$ are given by
\begin{eqnarray}
\text{int}(H_1)=\left\{ x\in\R_{\max}^n \mid \bigoplus_{i\in I}x_i<\bigoplus_{j\in J}x_j\oplus \alpha \right\},\label{eq:line1}\\
\text{int}(H_2)=\left\{ x\in\R_{\max}^n \mid \bigoplus_{i\in I}x_i>\bigoplus_{j\in J}x_j\oplus \alpha \right\}.\label{eq:line2}
\end{eqnarray}

If $L\not =\emptyset$, due to the continuity, we still have the inclusions:
\begin{eqnarray}
\left\{ x\in\R_{\max}^n \mid \bigoplus_{i\in I}x_i<\bigoplus_{j\in J}x_j\oplus \alpha \makebox{ and }
x_\ell=-\infty \makebox{ for } \ell\in L  \right\}\subseteq H_1,\label{eq:line1-infinity}\\
\left\{ x\in\R_{\max}^n \mid \bigoplus_{i\in I}x_i>\bigoplus_{j\in J}x_j\oplus \alpha  \makebox{ and }
x_\ell=-\infty \makebox{ for } \ell\in L  \right\}\subseteq H_2.\label{eq:line2-infinity}
\end{eqnarray}
By abuse of language, we continue to call the left hand sides in~\eqref{eq:line1-infinity},~\eqref{eq:line2-infinity} the \emph{interiors} of $H_1,H_2$.

The boundary of a hemispace $H$ is characterized by the linear equation~\eqref{e:-boundaryhalfs1} and it is a convex set. As the intersection of two convex sets is a convex set, both $H$ and $\complement H$ intersect the boundary of the halfspace in a convex set. Therefore, due to~\cite{BH}, in order to understand the geometric structure of a hemispace it is necessary to understand how to partition the boundary of a halfspace in two convex sets and then how to assign the pieces of the boundary to the interiors in order to complete a pair of complementary hemispaces.

Assume now that a hyperplane~\eqref{e:-boundaryhalfs1} bounds two complementary hemispaces. A first observation, discussed in Section~\ref{s:2}, is that the hyperplane has a \emph{conical decomposition in faces} (cones) of various dimensions. In Section~\ref{s:3} we prove two fundamental facts about the faces:

\begin{itemize}
\item if a face has a common point with a hemispace, then the whole face is included in the hemispace (Lemma~\ref{first-lemma});
\item if two different faces are included in a hemispace, then any segment joining those faces is included and intersects the union of the faces and their common boundary (Lemma~\ref{l:lemma2}).
\end{itemize}

We observe that the notion of face that we use can be extended to include the notion of \emph{(max-plus) sector}, related to the complement of a semispace, that appeared before in the max-plus literature~\cite{DS,J,NS1, NS2, NS3, KNS}. For a fixed supporting hyperplane, the conical decomposition of the boundary extends to a conical decomposition of the whole space $\R^n_{max}$. The main result, Theorem~\ref{t:main} presented in Section~\ref{s:main}, basically says that a pair of complementary hemispaces is given by a partition of the faces associated to a hyperplane in two collections, both satisfying a simple condition of closure. This is a new combinatorial result that was not observed in the previous work~\cite{BH, KNS}. Our result allows for an explicit counting and enumeration of hemispaces. In Section~\ref{s:counting} we count the number of hemispaces supported by hyperplanes centered in the origin which are strictly affine and nondegenerate, that is, for which the linear equation contains all variable and a free term. Theorem~\ref{t:counting773} shows that the number of such hemispaces is twice the $n$-th ordered Bell number. Finally, in Section~\ref{s:examples} we enumerate all hemispaces counted in Theorem~\ref{t:counting773} for $n=2$ and $n=3$ and show one more example of hemispace, supported by a hyperplane that is neither strictly affine nor degenerate.

\section{The combinatorial structure of the boundary}\label{s:2}

Without loss of generality, we can assume that the sets $I,J,L$ in~\eqref{e:-boundaryhalfs1} are ordered and the indices written in increasing order, that is, there exists $1\le m\le p\le q\le n$ such that $I=(1,\dots,m), J=(m+1,\dots,p), L=(p+1,\dots,q)$. If $\alpha=0$ appears in~\eqref{e:-boundaryhalfs1}, we denote $x_{n+1}:=0$ and denote $\bar J=(m+1,\dots,p,n+1).$ Otherwise denote $\bar J=J$.

It is clear that the equality may occur in~\eqref{e:-boundaryhalfs1} only when some terms on the left side of the linear equation are equal to some terms on the right side, and the rest of the terms in the linear equation are strictly smaller. Fixing the sets of indices for the equal terms naturally leads to the notion of "face". Nevertheless, it is convenient for our presentation to consider a more general notion of face, in which the maximal terms are among those appearing on a single side of~\eqref{e:-boundaryhalfs1}.

\begin{definition}\label{d:faces23}
Assume that a hyperplane $\mathcal{H}\subseteq \R^n_{\max}$ is defined by~\eqref{e:-boundaryhalfs1}. Let $0\le k\le n$. A \emph{$k$-codimensional face $F$ associated to $\mathcal{H}$} is a subset $F=F(I_F, \bar J_F)\subseteq \R^n_{\max}$ defined by two subsets $I_F\subseteq I, \bar J_F\subseteq \bar J$, where at least one of $I_F, \bar J_F$ is nonempty, $I_F\cup \bar J_F$ has $k+1$ elements and such that a point $x\in \R^n_{\max}$ belongs to $F$ if and only if:
\begin{equation}\label{eq:face-def}
\left \{
\begin{gathered}
x_{i_1}=x_{i_2}>-\infty, i_1,i_2\in I_F\cup \bar J_F,\\
x_k<x_i, i\in I_F\cup \bar J_F, k\in (I\cup \bar J)\setminus (I_F\cup \bar J_F),\\
x_{\ell}=-\infty, \ell\in L.
\end{gathered}
\right .
\end{equation}

A $k$-codimensional face of $\R^n_{\max}$  will be referred to as a \emph{$k$-face}. We denote $K_F=I_F\cup \bar J_F$ and call it the \emph{set of indices} of $F$. If $I_F\not =\emptyset, \bar J_F\not =\emptyset,$ then the face $F$ is also called \emph{pure face}.

In addition, if $\alpha$ is missing from~\eqref{e:-boundaryhalfs1}, we consider the extra face of type I characterized by:
\begin{equation}\label{eq:face-def2}
x_i=-\infty, i\in I\cup J\cup L,
\end{equation}
and if $L\not =\emptyset$ we consider the extra face of type II characterized by:
\begin{equation}\label{eq:face-def3}
x_i>-\infty, \text{ for some } i\in L.
\end{equation}

If  $\alpha$ is missing from~\eqref{e:-boundaryhalfs1} and $L\not =\emptyset$ we consider both extra faces~\eqref{eq:face-def2} and~\eqref{eq:face-def3}.

\end{definition}

The following lemma shows several properties of the collection of faces associated to a hyperplane and can be easily proved by inspection.

\begin{lemma}\label{l:somlem34} Assume that a hyperplane $\mathcal{H}\subseteq \R^n_{\max}$ is defined by~\eqref{e:-boundaryhalfs1}.
\begin{enumerate}
\item The collection of the sets of indices of the $k$-faces associated to $\mathcal{H}$ is closed under union.
\item Any two distinct faces associated to $\mathcal{H}$ are disjoint.
\item The union of all faces associated to $\mathcal{H}$ is $\R^n_{\max}$.
\end{enumerate}
\end{lemma}

\begin{remark} a) We observe that given a hyperplane $\mathcal{H}$ defined by~\eqref{e:-boundaryhalfs1}, the set of indices $K_F$ of a $k$-face $F$ uniquely determines $F$. If the linear equation of a hyperplane in $R^n_{\max}$ contains all variables $x_1, \dots, x_n, \alpha=0,$ and consequently $L=\emptyset,$ then the collection of faces related to $\mathcal{H}$, which are all $k$-faces, is indexed by $\mathcal{P}(n+1)\setminus \{\emptyset\},$ where $\mathcal{P}(n+1)$ is the collection of subsets of $[n+1]$.

b) The $k$-faces of minimal codimension appeared before in the literature under the name of sectors. They are related to the complements of semispaces~\cite{DS,J,NS1, NS2, NS3, KNS}.

c) It is easy to check that the faces are closed under (max-plus) addition and scalar multiplication by scalars different from $-\infty$, and therefore are also max-plus cones. It follows from Lemma~\ref{l:somlem34} that the collection of faces associated to a hyperplane gives a (max-plus) conical decomposition of $\R^n_{\max}$. We show in our main result that this decomposition is basically preserved by a partition of $\R^n_{\max}$ into a pair of complementary hemispaces. The only face that may be splitted further is the face of type I.

d) The pure faces are subsets of the hyperplane~\eqref{e:-boundaryhalfs1}. The faces that are not pure faces are included in the interiors of the halfspaces determined by~\eqref{e:-boundaryhalfs1}.
\end{remark}

We show now several examples of decompositions of $\R^n_{\max}$ in faces.

\begin{example}\label{e:example1} The faces in $\R^2_{\max}$ associated to the hyperplane $x_1\oplus x_2=0$, which give a conical decomposition of $\R^2_{\max}$, are characterized by:
\begin{enumerate}
\item $0$-codimensional faces, corresponding respectively to the sets of indices $\{1\},\{2\},\{3\}$:
\begin{enumerate}
\item $x_1>x_2, x_1>0$;
\item $x_2>x_1,x_2>0$;
\item $0>x_1,0>x_2$.
\end{enumerate}
\item $1$-codimensional faces, corresponding respectively to the sets of indices $\{1,2\},\{2,3\},\{1,3\}$:
\begin{enumerate}
\item $x_1=x_2>0$;
\item $x_2=0>x_1$; (pure)
\item $x_1=0>x_2$. (pure)
\end{enumerate}
\item $2$-codimensional face, corresponding to the set of indices $\{1,2,3\}$:
\begin{enumerate}
\item $x_1=x_2=0$. (pure)
\end{enumerate}
\end{enumerate}
\end{example}

\begin{example} The faces in $\R^3_{\max}$ associated to the hyperplane $x_1\oplus x_2=0$, which give a conical decomposition of $\R^3_{\max}$, are exactly those listed in Example~\ref{e:example1}. Note that the variable $x_3$ does not appear in the description of the faces.
\end{example}

\begin{example} The faces in $\R^2_{\max}$ associated to the hyperplane $x_1= x_2$, which give a conical decomposition of $\R^2_{\max}$, are characterized by:
\begin{enumerate}
\item $0$-codimensional faces, corresponding respectively to the sets of indices $\{1\},\{2\}$:
\begin{enumerate}
\item $x_1>x_2$;
\item $x_2>x_1$.
\end{enumerate}
\item $1$-codimensional face, corresponding respectively to the set of indices $\{1,2\}$:
\begin{enumerate}
\item $x_1=x_2>-\infty$.
\end{enumerate}
\item an extra face of type I due to the free term missing from the linear equation of the hyperplane:
\begin{enumerate}
\item $x_1=x_2=-\infty$.
\end{enumerate}
\end{enumerate}
\end{example}

\begin{example} The faces in $\R^2_{\max}$ associated to the hyperplane $x_1= 0, x_2=-\infty$, which give a conical decomposition of $\R^2_{\max}$, are characterized by:
\begin{enumerate}
\item $x_1=0,x_2=-\infty$.
\item $x_1>0,x_2=-\infty$.
\item $x_1<0,x_2=-\infty$.
\item $x_2>-\infty$.
\end{enumerate}
\end{example}

\begin{example} The faces in $\R^3_{\max}$ associated to the hyperplane $x_1= 0, x_2=-\infty, x_3=-\infty$, which give a conical decomposition of $\R^3_{\max}$, are characterized by:
\begin{enumerate}
\item $x_1=0,x_2=-\infty,x_3=-\infty$.
\item $x_1>0,x_2=-\infty,x_3=-\infty$.
\item $x_1<0,x_2=-\infty,x_3=-\infty$.
\item $x_2>-\infty,$ or $x_3>-\infty$.
\end{enumerate}
\end{example}

\begin{definition}
Assume that a hyperplane $\mathcal{H}\subseteq \R^n_{\max}$ is defined by~\eqref{e:-boundaryhalfs1}. Let $F_1=F_1(I_{F_1}, \bar J_{F_1}),F_2=F_2(I_{F_2}, \bar J_{F_2})$ be two $k$-faces of $\mathcal{H}$. Define the \emph{boundary subface} of $F_1,F_2$ to be the face of $\mathcal{H}$ defined by the set of indices
$I_{F_1}\cup I_{F_2}\cup \bar J_{F_1}\cup \bar J_{F_2}$.

We denote the boundary subface of $F_1,F_2$ by $\text{Bd}(F_1,F_2)$.
\end{definition}

\section{Some auxiliary lemmas}\label{s:3}

\begin{lemma}\label{first-lemma}
Let $\mathcal{H}$ be a hyperplane described by equation~\eqref{e:-boundaryhalfs1}, let $(H_1, H_2)$ be a pair of complementary hemispaces related to $\mathcal{H}$, and let $F$ be a $k$-face or a face of type II of $\mathcal{H}$.

Then $F\subseteq H_1$ or $F\subseteq H_2$.
\end{lemma}

\begin{proof} \emph{Assume first that $F$ is a $k$-face given by~\eqref{eq:face-def}.}

Throughout the proof, assume that the interior of $H_1$ satisfies~\eqref{eq:line1-infinity} and that the interior of $H_2$ satisfies~\eqref{eq:line2-infinity}.

Let $p=(p_i)_i, q=(q_i)_i \in F$. We will show that $p \in H_1$ if and only if $q \in H_1$ (which also gives us $p \in H_2$ if and only if $q \in H_2$ since $H_1$ and $H_2$ are complements).  To do so, we will find $r \in H_1$ and $s \in H_2$ such that either $p \in [q,r]$ and $p \in [q,s]$, or $q \in [p,r]$ and $q \in [p,s]$.  Since $H_1$ and $H_2$ are complements of each other and both convex, either of these scenarios establishes that $p \in H_1$ implies $q \in H_1$ and that $p \in H_2$ implies $q \in H_2$, which combine to yield the desired equivalence described above.  Observe that it suffices to consider the cases when $p$ and $q$ differ by only one coordinate or by a group of equal coordinates, since for any arbitrary $p,q \in F$ we could then go from $p$ to $q$ by a finite sequence of points $p, p_1, p_2, \ldots, p_{m-1}, p_m = q$ which changes only a coordinate or a group of equal coordinates at a time and preserves the property that $p \in H_1$ if and only if $p_i \in H_1$ ($i = 1, 2, \ldots, m$). Thus, $p$ and $q$ must be comparable and we can assume, without loss of generality, that $p > q$.

Throughout this process, the infinite coordinates $x_i, i\in L$ in  ~\eqref{eq:face-def} remain unchanged and we will ignore them in what follows. We will also ignore the coordinates that do not belong to $I\cup \bar J\cup L$, as they are not used in either of the equations~\eqref{eq:line1-infinity},~\eqref{eq:line2-infinity},~\eqref{eq:face-def}.

We first consider the case when $\alpha = 0$ and distinguish two subcases, $n+1 \notin \bar J_F$ and $n+1 \in \bar J_F$.

\emph{Case 1, $n+1 \notin \bar J_F$.}  Then $p_i := p_0 > 0$, $q_i := q_0 > 0$ for $i \in I_F \cup J_F$, and $p_i < p_0, q_i < q_0$ for $i \in (I \cup \bar J) \setminus (I_F \cup J_F)$.

If $p_0 > q_0 (>0)$ and all other coordinates of $p,q$ are equal, then $p_0 = q_0+c, c>0$.  Consider the point $r\in \R^n_{\max}$ of coordinates $r_i=q_0$ for $i \in J_F$,
$r_i=q_0 - 1$ for $i\in I_F$, and $r_i=q_i(=p_i)$ for $i\in (I\cup J)\setminus (I_F\cup J_F)$. Then $r$ satisfies equation~\eqref{eq:line1-infinity}, due to
\begin{equation*}
\bigoplus_{i\in I}r_i=\bigoplus_{i\in I\setminus I_F}q_i\oplus q_0-1<q_0 = \bigoplus_{j \in J}q_j \oplus0 = \bigoplus_{j\in J}r_j\oplus 0,
\end{equation*}
thus $r \in H_1$, and $q=r\oplus (-c)p$, which is a max-plus convex combination of $p,r$, so $q \in [p,r]$.  Now, consider the point $s\in \R^n_{\max}$ of coordinates $s_i=q_0$ for $i \in I_F$, $s_i = q_0-1$ for $i\in J_F$, and $s_i=q_i(=p_i)$ if $i\in (I\cup J)\setminus (I_F\cup J_F)$. Then $s$ satisfies equation~\eqref{eq:line2-infinity}, due to
\begin{equation*}
\bigoplus_{i\in I}s_i=\bigoplus_{i\in I}q_i=q_0 > \bigoplus_{j\in (J\setminus J_F)}q_j\oplus q_0-1 \oplus 0 = \bigoplus_{j\in J}s_j\oplus 0,
\end{equation*}
thus $s\in H_2$, and $q=s\oplus (-c)p$, which is a max-plus convex combination of $p,s$, so $q \in [p,s]$.  Hence we have found $r \in H_1$ and $s \in H_2$ such that $q \in [p,r] \cap [p,s]$, so we have that $p \in H_1$ if and only if $q \in H_1$.

Now, assume that $p_0 = q_0$ and $p,q$ differ on only one coordinate $k_0 \in (I \cup J) \setminus (I_F \cup J_F)$, so $p_{k_0} > q_{k_0}$.  Consider the point $r\in \R^n_{\max}$ of coordinates $r_i = p_0 (=q_0)$ for $i \in J_F$, $r_i = p_0-1(=q_0-1)$ for $i \in I_F$, $r_{k_0} = p_{k_0}$, and $r_i = p_i (=q_i)$ for $i \in (I \cup J \setminus \{k_0\})\setminus(I_F \cup J_F)$. Then $r$ satisfies equation~\eqref{eq:line1-infinity}, due to
\begin{equation*}
\bigoplus_{i\in I}r_i=\bigoplus_{i\in I \setminus I_F}p_i \oplus p_0-1< p_0 = \bigoplus_{j \in J}p_j \oplus0 = \bigoplus_{j \in J}r_j \oplus 0,
\end{equation*}
thus $r\in H_1$, and $p = r \oplus q$, which is a max-plus convex combination of $q,r$, so $p \in [q,r]$.  Now, consider the point $s\in \R^n_{\max}$ of coordinates $s_i=p_0(=q_0)$ for $i \in I_F$, $s_i=p_0-1(=q_0-1)$ for $i\in J_F$, $s_{k_0}=p_{k_0}$, and $s_i=p_i(=q_i)$ for $i\in (I\cup J\setminus \{k_0\})\setminus (I_F\cup J_F)$. Then $s$ satisfies equation~\eqref{eq:line2-infinity}, due to
\begin{equation*}
\bigoplus_{i\in I}s_i=\bigoplus_{i\in I}p_i=p_0>\bigoplus_{j\in J \setminus J_F}p_j \oplus p_0-1 \oplus 0 = \bigoplus_{j\in J}s_j \oplus 0,
\end{equation*}
thus $s\in H_2$, and $p = s \oplus q$, which is a max-plus convex combination of $q,s$, so $p \in [q,s]$.  Hence, $p \in H_1$ if and only if $q \in H_1$.

\emph{Case 2, $n+1 \in \bar J_F$.}  Then $p_i = q_i = 0$ for all $i \in I_F \cup \bar J_F$, and $p_i, q_i < 0$ for all $i \in (I \cup \bar J) \setminus (I_F \cup J_F)$.

Let the different coordinate of $p,q$ be $k_0 \in (I \cup J) \setminus (I_F \cup J_F)$, so $p_{k_o} > q_{k_0}$.  Consider the point $r\in \R^n_{\max}$ of coordinates $r_i=-1$ for $i \in I_F$, $r_{k_0}=p_{k_0}$, and $r_i=p_i(=q_i)$ for $i\in (I\cup J)\setminus (I_F\cup \{k_0\})$. Then $r$ satisfies equation~\eqref{eq:line1-infinity}, due to
\begin{equation*}
\bigoplus_{i\in I}r_i=\bigoplus_{i \in I \setminus I_F}p_i \oplus -1 < 0 = \bigoplus_{j \in J}p_j \oplus 0 = \bigoplus_{j \in J}r_j \oplus 0,
\end{equation*}
thus $r\in H_1$, and $p = r \oplus q$, which is a max-plus convex combination of $q,r$, so $p \in [q,r]$.  Now consider the point $s\in \R^n_{\max}$ of coordinates $s_i=1$ for $i \in I_F$, $s_{k_0}=p_{k_0}+1$ and $s_i=p_i(=q_i)$ if $i\in (I\cup J)\setminus (I_F \cup \{k_0\})$. Then $s$ satisfies equation~\eqref{eq:line2-infinity}, due to
\begin{equation*}
\bigoplus_{i\in I}s_i=1 > \bigoplus_{j\in J}s_j\oplus 0,
\end{equation*}
thus $s\in H_2$, and $p=(-1)s\oplus q$, which is a max-plus convex combination of $q,s$, so $p \in [q,s]$.  Hence, $p \in H_1$ if and only if $q \in H_1$.

For the case when $\alpha$ is missing from equation~\eqref{e:-boundaryhalfs1}, we can simply apply the proof for the case when $\alpha = 0$ and $n+1 \notin \bar J_F$, omitting the restriction that $p_0,q_0 > 0$ and any inclusions of $\alpha (= 0)$ in the quotations to equations ~\eqref{eq:line1-infinity} and ~\eqref{eq:line2-infinity}.

\emph{Assume now that $F$ is a face of type II given by~\eqref{eq:face-def3}.}

Let $p,q\in F$. It follows from~\eqref{eq:face-def3} that $p_i>-\infty$ for some $i\in L$ and $q_j>-\infty$ for some $j\in L$. Then any $z\in [p,q]$ has $z_{\ell}>-\infty$ for some $\ell\in L$, so $z\in F$. If $p\in H_1$ and $q\in H_2$, then there is a point in $[p,q]\subseteq F$ which belongs to their common boundary, $\mathcal{H}$. But from~\eqref{eq:face-def3} any point $z\in\mathcal{H}$ has $z_i=-\infty, i\in L,$ so $\mathcal{H}\cap F=\emptyset$.

We conclude that $F\subseteq H_1$ or $F\subseteq H_2$.
\end{proof}

The following lemma shows that an extra face of type I can be partitioned in two convex parts assigned to either hemispace in a pair of complementary hemispaces while preserving convexity.

\begin{lemma}\label{l:lemmaneginf}
Let $\mathcal{H}$ be a halfspace described by equation~\eqref{e:-boundaryhalfs1} and let $F_0$ be a face of type I. Let $F_0=F_1\cup F_2$ be a partition of $F_0$ in two convex sets. Let $(H_1, H_2)$ be a pair of complementary hemispaces related to $\mathcal{H}$ such that $F_1\subseteq H_1$ and $F_2\subseteq H_2$.  Then $\left ((H_1 \setminus F_1)\cup F_2, (H_2 \setminus F_2)\cup F_1\right )$ is a pair of complementary hemispaces related to $\mathcal{H}$ as well.
\end{lemma}

\begin{proof} We observe that due to Lemma~\ref{first-lemma} the faces associated to $\mathcal{H}$ that are different from $F_0$ belong to only one of the hemispaces. Consequently, $H_1\setminus F_1$ and $H_2\setminus F_2$ are unions of faces. As each face is a cone, it follows that $p\in H_1\setminus F_1$ or $p\in H_2$ implies $\alpha p\in H_1\setminus F_1,$ respectively $\alpha p\in H_2\setminus F_2,$ for each $\alpha>-\infty$. As each face associated to $\mathcal{H}$ is defined only by the coordinates in $I\cup J\cup L$, it follows that if $p\in H_i\setminus F_i, i=1,2,$ and $(\alpha p)_i, i\in I\cup J\cup L,$ for some $\alpha>-\infty,$ then $\alpha p\in H_i\setminus F_i, i=1,2$.

Due to the symmetry, it is enough to show that $(H_1 \setminus F_1)\cup F_2$ is convex. Let $p,q\in (H_1 \setminus F_1)\cup F_2$. To prove that $[p,q]\subseteq (H_1 \setminus F_1)\cup F_2,$ it is enough to show the following:
\begin{enumerate}
\item $p,q\in H_1\setminus F_1$ implies $[p,q]\subseteq (H_1\setminus F_1)\cup F_2$;
\item $p\in H_1\setminus F_1, q\in F_2$ implies $[p,q]\subseteq (H_1\setminus F_1)\cup F_2$;
\item $p,q\in F_2$ implies $[p,q]\subseteq (H_1\setminus F_1)\cup F_2$.
\end{enumerate}

1. Clearly $[p,q]\subseteq H_1$ as $p,q\in H_1$ and $H_1$ is a convex set. As $H_1, H_2$ are disjoint and $F_2\subseteq H_2,$ one has $[p,q]\cap F_2=\emptyset$.

If $p,q\not \in F_1$, then $p,q\not \in F_0$ and it follows from~\eqref{eq:face-def2} that $p_i\not =-\infty, q_j\not =-\infty,$ for some $i,j\in I\cup J\cup L,$ which implies that for any point $z\in [p,q]$ one has $z_i>-\infty$  for some $i\in I\cup J\cup L,$ thus $z\not \in F_0$ and $[p,q]\subseteq (H_1\setminus F_1)\cup F_2$.

2. Let $z = \alpha p \oplus \beta q \in [p, q]$, $\alpha\oplus\beta=0$.  If $\beta = 0$ then $z = \alpha p\oplus q$. If $\alpha=-\infty$, then $z=q\in F_2$. If $\alpha\not =-\infty,$ then $z_i=(\alpha p)_i, i\in I\cup J\cup L,$ and from the observation at the beginning of the proof it follows that $z\in H_1\setminus F_1$. If $\alpha = 0$, then $z = p\oplus \beta q$. As the coordinates $z_i, i\in I\cup J\cup L,$ coincide with the corresponding coordinates of $p$, it follows that $z\in H_1\setminus F_1$.

3. This follows right away from the fact that $F_2$ is a convex set.
\end{proof}

The following lemma shows that an extra face of type II can be assigned to either hemispace in a pair of complementary hemispaces preserving convexity.

\begin{lemma}\label{l:lemmaneginf2}
Let $\mathcal{H}$ be a halfspace described by equation~\eqref{e:-boundaryhalfs1} and let $F$ be a face of type II. Let $(H_1, H_2)$ be a pair of complementary hemispaces related to $\mathcal{H}$ such that $F\subseteq H_1$.  Then $(H_1 \setminus F, H_2\cup F)$ is a pair of complementary hemispaces related to $\mathcal{H}$ as well.
\end{lemma}

\begin{proof}
We show that $H_1 \setminus F$ is convex. Indeed if $p,q\in H_1 \setminus F$, then clearly $[p,q]\in H_1$ as $H_1$ is convex. Moreover, it follows from~\eqref{eq:face-def3} that $p_i=-\infty, q_i=-\infty,i\in L,$ which implies that for any point $z\in [p,q]$ one has $z_i=-\infty,i\in L,$ thus $z\not \in F$ and $z\in H_1 \setminus F$.

It remains to show that $H_2 \cup F$ is convex. Let $p \in H_2, q\in F$ be fixed. The other cases follow immediately from the fact that $H_2$ and $F$ are convex. Let $z = \alpha p \oplus \beta q \in [p, q]$.  If $\beta = 0$ then $z = \alpha p\oplus q$. Then $z_i\ge q_i, i\in L,$ so from~\eqref{eq:face-def3} at least one of $z_i, i\in L,$ is greater then $-\infty$, so $z\in F\subseteq H_2\cup F$. If $\alpha = 0$, then $z = p\oplus \beta q$. If $\beta=-\infty,$ one has $z=p$, so $z\in H_2\subseteq H_2\cup F$. If $\beta\not =-\infty,$ then $\beta q$ has at least a coordinate $(\beta q)_i, i\in L,$ that is greater then $-\infty$, so it follows that $z_i>-\infty$ and therefore $z\in F\subseteq H_2\cup F$.
\end{proof}

\begin{lemma}\label{l:lemma2} Let $\mathcal{H}\subseteq \R^n_{\max}$ be a max-plus hyperplane given by~\eqref{e:-boundaryhalfs1}.
Let $F_1,F_2$ be $k$-faces of $\mathcal{H}$ and $F_3=\text{Bd}(F_1,F_2)$. Let $x\in F_1, y\in F_2$. Then the segment $[x,y]$ is included in the union $F_1\cup F_2\cup F_3$, and $[x,y]\cap F_i\not =\emptyset, 1\le i\le 3.$
\end{lemma}

\begin{proof} Let $F_1=F_1(I_{F_1},\bar J_{F_1}), F_2=F_2(I_{F_2},\bar J_{F_2})$. Then
\begin{equation}\label{eq:face-point1}
\begin{gathered}
p:=x_{i_1}=x_{i_2}, i_1,i_2\in I_{F_1}\cup \bar J_{F_1},\\
x_k<x_i, i\in I_{F_1}\cup \bar J_{F_1}, k\in (I\cup \bar J)\setminus (I_{F_1}\cup \bar J_{F_1}),\\
x_{\ell}=-\infty, \ell\in L
\end{gathered}
\end{equation}
and
\begin{equation}\label{eq:face-point2}
\begin{gathered}
q:=y_{i_1}=y_{i_2}, i_1,i_2\in I_{F_2}\cup \bar J_{F_2},\\
y_k<y_i, i\in I_{F_2}\cup \bar J_{F_2}, k\in (I\cup \bar J)\setminus (I_{F_2}\cup \bar J_{F_2}),\\
y_{\ell}=-\infty, \ell\in L.
\end{gathered}
\end{equation}

Without loss, assume $p \leq q$. Let $z=\alpha x\oplus \beta y\in [x,y]$. If $\beta=0$ then $z \in F_2$ except if $p=q$ and $\alpha = 0$, in which case $z \in F_3$.  If $\alpha=0$ then
$z\in F_1$ for $\beta<p-q$, $z\in F_3$ for $\beta=p-q$, and $z\in F_2$ for $p-q<\beta\leq 0$ (where this case does not occur if $p=q$).
\end{proof}

\section{The main result}\label{s:main}

In this section we give a geometric description for a hemispace related to a hyperplane.

It follows from Lemma~\ref{first-lemma} that if $(H_1,H_2)$ is a pair of complementary hemispaces related to a hyperplane $\mathcal{H}$, then any $k$-face of $\mathcal{H}$ is included either in $H_1$ or in $H_2$. The following theorem describes the partition of the $k$-faces among $H_1$ and $H_2$.

\begin{theorem}\label{t:main} Let $(H_1,H_2)$ be a pair of complementary hemispaces related to the hyperplane $\mathcal{H}$ given by~\eqref{e:-boundaryhalfs1}. Let $\mathcal{F}$ be the set of $k$-faces of $\mathcal{H}$, $\mathcal{F}_1 \subseteq \mathcal{F}$ be the set of $k$-faces included in $H_1$ and  $\mathcal{F}_2 \subseteq \mathcal{F}$ be the set of $k$-faces included in $H_2$. Let $\mathcal{K}_1$ be the collection of sets of indices of $\mathcal{F}_1$ and $\mathcal{K}_2$ be the collection of sets of indices of $\mathcal{F}_2$. Then $\mathcal{K}_1$ and $\mathcal{K}_2$ are each closed under union. Moreover, $\mathcal{P}(I)$ is a subset of one and $\mathcal{P}(\bar J)$ is a subset of the other.

Conversely, assume that the hyperplane $\mathcal{H}$ is given by~\eqref{e:-boundaryhalfs1}. Then any partition of the set $\mathcal{F}$ of $k$-faces of $\mathcal{H}$  in two subsets $\mathcal{F}_1, \mathcal{F}_2$ with collections of sets of indices $\mathcal{K}_1, \mathcal{K}_2$ closed under union
determines a pair of disjoint convex sets
given by $\cup_{F\in \mathcal{F}_1}F$ and $\cup_{F\in \mathcal{F}_2}F$. Moreover, if $\mathcal{P}(I)$ is a subset of $\mathcal{K}_1$ and $\mathcal{P}(\bar J)$ is a subset of $\mathcal{K}_2$, then there exists a pair $(H_1,H_2)$ of complementary hemispaces related to $\mathcal{H}$, such that one convex sets above is a subset of $H_1$ and the other is a subset of $H_2$.
\end{theorem}

\begin{proof} Let $K_F, K_G\in \mathcal{K}_1$ be the sets of indices for the $k$-faces $F,G \subseteq H_1$. It follows from Lemma~\ref{l:lemma2} that any segment with endpoints in $F$ and $G$ intersects $\text{Bd}(F,G)$. As $H_1$ is convex, $\text{Bd}(F,G)$ has to intersect $H_1$. It follows now from Lemma~\ref{first-lemma} that the face $\text{Bd}(F,G)$ is included in $H_1$. As the set of indices of $\text{Bd}(F,G)$ is $K_F\cup K_G$, we have $K_F\cup K_G\in\mathcal{K}_1$. The partition of the singletons and the closure under the union forces $\mathcal{P}(I)$ to be a subset of one of $\mathcal{K}_i$ and $\mathcal{P}(\bar J)$ to be a subset of the other.

Conversely, assume that the collections of sets of indices $\mathcal{K}_1, \mathcal{K}_2$ are closed under union. We must show that the (obviously disjoint) unions $\cup_{F\in \mathcal{F}_1}F, \cup_{F\in \mathcal{F}_2}F$ are both convex. Let $x,y\in \cup_{F\in \mathcal{F}_1}F$. If $x,y$ belong to the same $k$-face $F$, then $[x,y]\subseteq F$, as $F$ is convex. If $x\in F_1, y\in F_2$, $F_1,F_2$ different $k$-faces in  $\mathcal{F}_1$, then it follows from Lemma~\ref{l:lemma2} that the segment $[x,y]$ is the concatenation of three parts, one included in $F_1$, one in $F_2$, and one included in $\text{Bd}(F_1,F_2)$. Due to the closeness under union of $\mathcal{K}_1$, one has $\text{Bd}(F_1,F_2)\in\mathcal{F}_1$. So $[x,y]\subseteq \cup_{F\in \mathcal{F}_1}F$. It follows that $\cup_{F\in \mathcal{F}_1}F, \cup_{F\in \mathcal{F}_2}F$ are both convex. The fact that there exists a pair of complementary hemispaces separating these disjoint convex sets is immediate from Stone-Kakutani theorem. The extra condition $ \mathcal{P}(\bar J)\subseteq \mathcal{K}_1, \mathcal{P}(I)\subseteq \mathcal{K}_2$ forces the assignment of the singletons and implies that the pair of hemispaces is related to $\mathcal{H}$.
\end{proof}

We are ready to describe the geometric structure of a general pair of complementary hemispaces associated to a hyperplane.

\begin{theorem}\label{t:main2} Let $\mathcal{H}$ be a hyperplane given by~\eqref{e:-boundaryhalfs1}. Let $\mathcal{F}$ be the set of $k$-faces associated to $\mathcal{H}$, with the collection of the sets of indices $\mathcal{K}$, $F_I$ be the face of type I (if any), and $F_{II}$ be the face of type II (if any). A decomposition of $\R^n_{\max}$ in a pair of complementary hemispaces $(H_1, H_2)$ related to $\mathcal{H}$ is obtained in the following way:
\begin{itemize}
\item take a partition of $\mathcal{K}$ into two families closed under union, with corresponding sets of faces $\mathcal{F}_1, \mathcal{F}_2$;
\item take a partition of $F_I$ into two convex sets $F_I^1, F_I^2$;
\end{itemize}
and then define
\begin{equation}
\begin{aligned}
H_1&=\left (\cup_{F\in \mathcal{F}_1}F\right )\cup F_I^1\cup F_{II};\\
H_2&=\left (\cup_{F\in \mathcal{F}_2}F\right )\cup F_I^2.
\end{aligned}
\end{equation}
\end{theorem}

\begin{proof} The result is a consequence of Theorem~\ref{t:main}, Lemma~\ref{l:lemmaneginf} and Lemma~\ref{l:lemmaneginf2}.
\end{proof}

\begin{corollary} Let $H\subseteq \R^n_{\max}$ be a proper hemispace. Then $H$ is a disjoint union of a finite set of cones.
\end{corollary}

\begin{proof} Let $\mathcal{H}\subseteq \R^n_{\max}$ be a hyperplane supporting $H$. If there is no face of type I associated to $\mathcal{H}$, then it follows from Theorem~\ref{t:main2} that $H$ is a disjoint union of faces. As each face is a cone, the corollary follows. If there is a face of type I, then observe first that a face of type I is a semimodule and, moreover, it is isomorphic to a semimodule $R^d_{\max}, d<n.$ Then use again Theorem~\ref{t:main2} and mathematical induction.
\end{proof}

\section{Counting the hemispaces with a fixed finite center}\label{s:counting}

Our main result allows for an explicit counting and enumeration of hemispaces of certain type.
\begin{theorem}\label{t:counting773}
In $\R^n_{\max}$ there are exactly $2f(n)$ hemispaces (including the whole space and the empty set) related to a strictly affine nondegenerate hyperplane $\mathcal{H}$ centered in the origin, where $f(0) = 1$ and $f(n), n\ge 1,$ satisfies the recurrence formula:
\begin{equation}\label{eq:count-hyp}
f(n) = \binom{n+1}{1}f(n-1) + \binom{n+1}{2}f(n-2) + \cdots + \binom{n+1}{n}f(0)+1.
\end{equation}
\end{theorem}

\begin{proof} As $\mathcal{H}$ is a strictly affine nondegenerate hyperplane in $\R^n_{\max}$, we need to work with the variables $x_1, x_2,\dots,x_{n+1}$. All faces that appear are $k$-faces. Theorem~\ref{t:main} reduces the counting of proper hemispaces to a combinatorial problem. Let $\mathcal{P}(n+1)$ be the collection of subsets of a set with $n+1$ elements. We need to determine the number of partitions of $\mathcal{P}(n+1)$ in two nonempty collections of subsets that are closed under union. We prove by induction on $n$ that $f(n)$ gives the number of hemispaces that contain the origin. To count their complements as well, we multiply by 2.

We check the induction hypothesis for $n=1$. The hemispaces in $\R_{\max}$ supported by a nondegenerate strictly affine line centered in the origin that, in addition, contain the origin are the closed positive and negative half-lines, and the whole line, so clearly $f(1) = \binom{1+1}{1}f(0)+1=3$, as predicted by \eqref{eq:count-hyp}.

Let $(H, \complement H)$ be a pair of complementary hemispaces in $\R^n_{\max}$ related to $\mathcal{H}$.
Assume that $H$ contains the origin.  It follows that at most one $(n-1)$-codimensional face in $\mathcal{H}$, which has a set of indices of cardinality $n$, can be in $\complement H$. Indeed, having two $(n-1)$-codimensional faces in $\complement H$ would also require $\complement H$ to contain the origin due to Theorem~\ref{t:main}. There are $\binom{n+1}{1}$ possible $(n-1)$-codimensional faces in $\mathcal{P}(n+1)$. Once an $(n-1)$-codimensional face $F$ is chosen to be in $\complement H$, every face $\tilde F$ that does not contain in the set of indices $K_{\tilde F}$ the variable not maximal in $F$, say $x_{i_0}$, must be in $H$. Indeed, having such a face included in $\complement H$ forces $0\in \complement H$ due to Theorem~\ref{t:main}, in contradiction with the assumption above.

It remains to partition the sets of indices which do not contain $x_{i_0}$. This is an equivalent partition problem in $\R^{n-1}_{\max}$. There are $\binom{n+1}{1}$ ways to reduce the problem to a $(n-1)$-dimensional problem. As the number of hemispaces in $\R^{n-1}_{\max}$ is $f(n-1)$,  this accounts for the $\binom{n+1}{1}f(n-1)$ term in formula~\eqref{eq:count-hyp}.

Alternatively, $H$ could contain all $(n-1)$-codimensional faces.  In this case, there is no more than one $(n-2)$-codimensional face in $\complement H$ since, by Theorem~\ref{t:main}, having two distinct $(n-2)$-codimensional faces in $\complement H$ would also require that either the origin or an $(n-1)$-codimensional face to be in $\complement H$ due to Theorem~\ref{t:main}.  Thus, there are $\binom{n+1}{2}$ $(n-2)$-codimensional faces that can be chosen to be in $\complement H$, and doing so reduces the problem to $\R^{n-2}_{\max}$.  This accounts for the $\binom{n+1}{2}f(n-2)$ term in formula~\eqref{eq:count-hyp}.  Again, we could also choose for every $(n-2)$-codimensional face to be in $H$.  Following this pattern for $1\le d\le n-1$, it is clear that there are $\binom{n+1}{d}$ ways to reduce the problem to the case of $\R^{n-d}_{\max}$, which accounts for all the terms in formula~\eqref{eq:count-hyp}, except the last term $+1$ which represents the whole space $\R^n_{\max}$.
\end{proof}

\begin{remark} The first six terms in the sequence $f(n)$ are: 1, 3, 13, 75, 541, 4283. The sequence is listed in the On-line Encyclopedia of Integer Sequences~\cite{OEIS} as A000670 and has many well known combinatorial interpretations. Among others, $f(n)$ counts the number of ways $n$ competitors can rank in a competition, allowing for the possibility of ties, or, equivalently, the number of weak orders on a set with $n$ elements. The terms of the sequence $f(n)$ are usually called \emph{ordered Bell numbers} or \emph{Fubini numbers}. The sequence of ordered Bell numbers have a growth rate much higher than $n!$. Ordered Bell numbers also count permutohedron faces, Cayley trees, Cayley permutations, ordered multiplicative partitions of square free numbers, and equivalent formulae in Fubini's theorem. Our result shows a bijective correspondence between these combinatorial objects and max-plus hemispaces. A reference mentioning the formula appearing in Theorem~\ref{t:counting773} is the paper of Gross~\cite{gross}.
\end{remark}

A corollary of Theorem~\ref{t:counting773} is the following combinatorial result, which we did not find in the literature.

\begin{corollary} The number of ways to split $\mathcal{P}(n),$ the collection of subsets of a set with $n$ elements, into two subcollections closed under the union, not considering their order, is given by the ordered Bell number $f(n)$.
\end{corollary}

\begin{proof} The statement follows from the proof of Theorem~\ref{t:counting773}. We only describe the bijective correspondence between the family of weak orders on $[n]$ and the splittings of $\mathcal{P}(n)$. To a given weak order on $[n]$ we will associate the subcollection $\mathcal{C}$ of $\mathcal{P}(n)$ that contains $[n]$. One can think at a weak order on $[n]$ as a pyramid containing all elements of $[n]$ arranged in $k$ layers. All elements in a layer are assumed to be equal and elements in a higher layer are strictly larger then the elements in a lower layer. At the first step, we include in $\mathcal{C}$ all subsets containing one element of the first layer and all possible unions of such sets. At the second step, we include in $\mathcal{C}$ all subsets containing one element from the first and one element from the second layer, and all possible unions of such sets. At the third step, we include in $\mathcal{C}$ all subsets containing one element of the first, one element of the second layer and one element of the third layer, and all possible unions of such sets. We continue like this until we reach the $k$-th layer and finally have $[n]$ included in $\mathcal{C}$. We observe that the construction of $\mathcal{C}$ guarantees that the subcollection $\mathcal{P}(n)\setminus \mathcal{C}$ is closed under the union as well.
\end{proof}

\begin{remark} In a related direction, ordered Bell numbers also appear when counting the types of semispaces in max-min (or fuzzy) algebra~\cite{NS-maxmin}.
\end{remark}

\section{Examples of hemispaces}\label{s:examples}

\begin{example} We list the hemispaces in $\R_{\max}^2$ related to a strictly affine hyperplane centered in the origin. We use the method provided by Theorem~\ref{t:main}. We work with 3 coordinates, 1, 2, 3 which are partitioned in two nonempty sets $I, \bar J$. There are 3 ways to partition into sets of cardinality 1 and 2. Assume that the partition is $\{\{1,2\},\{3\}\}$. The other two cases are similar. The equation of the hyperplane $\mathcal{H}$ is $x_1\oplus x_2=0$. Assume $I=\{1,2\}, \bar J=\{3\}$. Let $H_1, H_2$ be a pair of complementary hemispaces related to $\mathcal{H}$. We already have the faces $\{1\},\{2\},\{1,2\}$ in $H_2$ and the face $\{3\}$ in $H_1$. It remains to partition the faces $\{1,3\},\{2,3\},\{1,2,3\}$ such that the closure under the union required in Theorem~\ref{t:main} holds. This can be done in 4 ways:
\begin{enumerate}
\item $\{1,3\},\{2,3\},\{1,2,3\}$ are in $H_2$, thus overall
\begin{itemize}
\item  $\{3\}$ is in $H_1$;
\item  $\{1\},\{2\},\{1,2\},\{1,3\},\{2,3\},\{1,2,3\}$ are in $H_2$;
\end{itemize}
\item $\{2,3\},\{1,2,3\}$ are in $H_2$ and $\{1,3\}$ is in $H_1$, thus overall
\begin{itemize}
\item  $\{3\}, \{1,3\}$ are in $H_1$;
\item  $\{1\},\{2\},\{1,2\},\{2,3\},\{1,2,3\}$ are in $H_2$;
\end{itemize}
\item $\{1,3\},\{1,2,3\}$ are in $H_2$ and $\{2,3\}$ is in $H_1$, thus overall
\begin{itemize}
\item  $\{3\}, \{2,3\}$ are in $H_1$;
\item  $\{1\},\{2\},\{1,2\},\{1,3\},\{1,2,3\}$ are in $H_2$;
\end{itemize}
\item $\{1,3\},\{2,3\},\{1,2,3\}$ are in $H_1$, thus overall
\begin{itemize}
\item  $\{3\}, \{1,3\},\{2,3\},\{1,2,3\}$ are in $H_1$;
\item $\{1\},\{2\},\{1,2\}$ are in $H_2$.
\end{itemize}
\end{enumerate}
The list above gives 4 pairs of complementary hemispaces, hence 8 distinct proper hemispaces. Considering the other two cases for the partition $I,\bar J$, overall there are $8\times 3=24$ distinct proper hemispaces. If we add to these the empty set and $\R^2_{\max}$, which correspond to the partition $\emptyset$ in $H_1$ and $\mathcal{P}(3)\setminus \{\emptyset\}$ in $H_2$, we have 26 hemispaces, as predicted by Theorem~\ref{t:counting773}.
\end{example}

\begin{example} We briefly list the hemispaces in $\R_{\max}^3$ related to a strictly affine hyperplane centered in the origin. We use the method provided by the proof of Theorem~\ref{t:main}. We work with 4 coordinates, 1, 2, 3, 4. The elements of $\mathcal{P}(4)$ are partitioned between a pair of complementary hemispaces $H_1, H_2$. We assume that the origin, which has the set of indices $\{1,2,3,4\},$ belongs to $H_1$. We look now at the partition of the $2$-codimensional faces, which correspond to the sets of indices $\{1,2,3\},\{1,2,4\},\{1,3,4\},\{2,3,4\}$. As the origin belongs to $H_1$, at most one $2$-codimensional face can belong to $H_2$. We distinguish two cases:

\begin{enumerate}
\item all $2$-codimensional faces belong to $H_1$;
\item one $2$-codimensional face, say $\{1,2,3\}$, belongs to $H_2$, and the others belong to $H_1$.
\end{enumerate}

\emph{Case 1.} We partition now the $1$-codimensional faces, which are in number of $6$: $\{1,2\},\{1,3\},\{1,4\},\{2,3\},\{2,4\},\{3,4\}$. At most one of these can belong to $H_2$, because otherwise $H_2$ would include a $2$-codimensional face, in contradiction to our assumption. We distinguish two subcases:

\begin{enumerate}
\item all $1$-codimensional faces belong to $H_1$;
\item one $1$-codimensional face, say $\{1,2\},$ belongs to $H_2$, and the others belong to $H_1$.
\end{enumerate}

\emph{Subcase 1.} We partition now the $0$-codimensional faces, which are in number of $4$: $\{1\},\{2\},\{3\},\{4\}.$ At most one of these can belong to $H_2$, because otherwise $H_2$ would include a $1$-codimensional face, in contradiction to our assumption. We distinguish two subsubcases:

\begin{enumerate}
\item all $0$-codimensional faces belong to $H_1$;
\item one $0$-codimensional face, say $\{1\},$ belongs to $H_2$, and the others belong to $H_1$.
\end{enumerate}

There are $2\times 5=10$ distinct hemispaces in Subcase 1. Note that we multiply by 2 in order to include in the counting the complements of the hemispaces containing the origin.

\emph{Subcase 2.} As above, we assume that $\{1,2\},$ belongs to $H_2$, and the other $1$-codimensional faces belong to $H_1$. It remains to partition the $0$-codimensional faces, which are in number of $4$: $\{1\},\{2\},\{3\},\{4\}.$ The faces $\{3\},\{4\}$ should be assigned to $H_1$, because otherwise $H_2$ contains a $2$-codimensional face, in contradiction to our assumption. Also, $\{1\},\{2\}$ cannot be both in $H_1$, because this would imply that $\{1,2\},$ belongs to $H_1$, again in contradiction to our assumptions. We are left with two subsubcases:

\begin{enumerate}
\item $\{1\},\{2\}$ belong to $H_2$;
\item only one of $\{1\},\{2\}$ belongs to $H_2$, and the other belongs to $H_1$.
\end{enumerate}

There are $6\times 2\times 3=36$ distinct hemispaces in Subcase 2.

Overall, there are $10+36=46$ distinct hemispaces in Case 1.

\emph{Case 2.} We partition now the $1$-codimensional faces, which are in number of $6$: $\{1,2\},\{1,3\},\{1,4\},\{2,3\},\{2,4\},\{3,4\}$. All faces containing the variable $4$ are in $H_1$, because otherwise $H_2$ contains the origin, in contradiction with our assumption. We cannot have two of the remaining $1$-codimensional faces in $H_1$ because this implies that $\{1,2,3\}$ is in $H_1$, again in contradiction to our assumption. We are left with two subcases:

\begin{enumerate}
\item $\{1,2\},\{1,3\},\{2,3\}$ belong to $H_2$;
\item one of $\{1,2\},\{1,3\},\{2,3\}$ belongs to $H_1$.
\end{enumerate}

\emph{Subcase 1.} It remains to partition the $0$-codimensional faces. We observe first that $\{4\}$ belongs to $H_1$, because otherwise the origin would be in $H_2$, in contradiction with our assumption. Also, it is not possible for two of the remaining variables to be in $H_1$, because this implies that one of the $1$-codimensional faces $\{1,2\},\{1,3\},\{2,3\}$ belongs to $H_1$, in contradiction with our assumption. We are left with two subsubcases:

\begin{enumerate}
\item exactly one of $\{1\},\{2\},\{3\}$ belongs to $H_1$;
\item $\{1\},\{2\},\{3\}$ belong to $H_2$.
\end{enumerate}

There are $4\times 4\times 2=32$ distinct hemispaces in Subcase 1.

\emph{Subcase 2.} In order to fix the notation, assume that $\{1,2\}$ belongs to $H_1$. Observe first that the faces $\{3\},\{4\}$ belong to $H_2$, as otherwise one of the faces $\{1,2,3\}$ or $\{1,2,4\}$ belongs to $H_1$, in contradiction to our assumption. Also, we cannot have both $\{1\},\{2\}$ in $H_2$, as this implies that $\{1,2\}$ belongs to $H_2$, in contradiction with our assumption. We are left with two subcases:

\begin{enumerate}
\item exactly one of $\{1\},\{2\}$ belongs to $H_2$;
\item $\{1\},\{2\}$ belong to $H_1$.
\end{enumerate}

There are $4\times 3\times 3\times 2=72$ distinct hemispaces in Subcase 2.

Overall there are $32+72=104$ distinct hemispaces in Case 2.

Adding the number of cases that appear in Case 1 and Case 2 we have 150 distinct hemispaces, as predicted by Theorem~\ref{t:counting773}.
\end{example}

\begin{example} We show a pair of complementary hemispaces related to a degenerate hyperplane. Consider the degenerate hyperplane in $\R^4_{\max}$ with equation $x_1=x_2, x_3=-\infty$. The $k$-faces are indexed by $\{1\},\{2\},\{1,2\}$, and denoted by $F_{\{1\}}, F_{\{2\}}, F_{\{1,2\}}$. We have an extra face of type I given by $F_I=\{x\in \R^4_{\max}\vert x_1=x_2=x_3=-\infty\}$ and an extra face of type II given by $F_{II}=\{x\in \R^4_{\max}\vert x_3>-\infty\}$. We may split the face of type I in two convex parts, say
\begin{equation*}
\begin{aligned}
F_I^1&=\{x\in \R^4_{\max}\vert x_1=x_2=x_3=-\infty, x_4\ge 0\},\\
F_I^2&=\{x\in \R^4_{\max}\vert x_1=x_2=x_3=-\infty, x_4< 0\}.
\end{aligned}
\end{equation*}

It follows from Theorem~\ref{t:main2} that the following assignment of the faces gives a pair $(H_1,H_2)$ of complementary hemispaces related to the hyperplane $\mathcal{H}$:
\begin{equation*}
\begin{aligned}
H_1=&\{F_{\{1\}}, F^1_I, F_{II}\},\\
H_2=&\{F_{\{2\}}, F_{\{1,2\}}, F^1_I\}.
\end{aligned}
\end{equation*}
\end{example}

\section*{Acknowledgement}

This paper was written during the Summer 2013 REU program at Pennsylvania State University, supported by the NSF grant  DMS-0943603. D. Ehrmann, Z. Higgins are undergraduate students. V. Nitica was one of the faculty coordinators. He was supported by Simons Foundation Grant 208729.

\end{document}